\documentclass[a4paper]{amsart}
\usepackage{wasowicz_hyperref}
\usepackage{times}

\newcommand{\m}{\mathbf{m}}
\newcommand{\QQ}{\mathcal{Q}}
\newcommand{\K}{\mathcal{K}}
\newcommand{\M}{\mathcal{M}}
\newcommand{\B}{\mathcal{B}}
\renewcommand{\P}{\mathcal{P}}

\newcommand{\od}[1][]{\bigl([-1,1]#1\bigr)}
\newcommand{\oda}{\od[,a]}
\newcommand{\odb}{\od[,b]}
\newcommand{\tmu}{\widetilde{\mu}}

\renewcommand{\dx}[1][x]{\mathrm{d}#1}

\DeclareMathOperator{\ext}{ext}
\DeclareMathOperator{\kw}{Quad_+}
\DeclareMathOperator{\D}{Disc}

\info{Published: J. Math. Anal. Appl. 423~(2015), 1838--1848\\ DOI: \url{http://dx.doi.org/10.1016/j.jmaa.2014.11.001}}
\author{Teresa Rajba}
\address{\SWaddr}
\email{\color{blue}trajba@ath.bielsko.pl}
\author{\SW}
\email{\SWmail}
\title{Extremal measures with prescribed moments}
\date{}

\begin{document}
\begin{abstract}
 In the approximate integration some inequalities between the quadratures and the integrals approximated by them are called \emph{extremalities}. On the other hand, the set of all quadratures is convex. We are trying to find possible connections between extremalities and extremal quadratures (in the sense of extreme points of a~convex set). Of course, the quadratures are the integrals \wrt~discrete measures and, moreover, a~quadrature is extremal if and only if the associated measure is extremal. Hence the natural problem arises to give some description of extremal measures with prescribed moments in the general (not only discrete) case. In this paper we deal with symmetric measures with prescribed first four moments. The full description (with no symmetry assumptions, and/or not only four moments are prescribed and so on) is far to be done.
\end{abstract}
\maketitle

\section{Introduction}

The second-named author considered in~\cite{Was10MIA} so-called \emph{extremalities} in the approximate integration.

Let $P_n$ be the $n$-th degree Legendre polynomial given by the Rodrigues formula
\[
 P_n(x)=\frac{1}{2^n n!}\frac{\dx[]^n}{\dx^n}(x^2-1)^2\,.
\]
 Then $P_n$ has $n$ distinct roots $x_1,\dots,x_n\in(-1,1)$. The $n$-point Gauss--Legendre quadrature is the positive linear functional on $\R^{[-1,1]}$ given by
\[
 \G_n[f]=\sum_{i=1}^nw_if(x_i)
\]
with the weights
\[
 w_i=\frac{2(1-x_i^2)}{(n+1)^2P_{n+1}^2(x_i)}\,,\quad i=1,\dots,n\,.
\]
The $(n+1)$-point Lobatto quadrature is the functional
\[
 \Lob_{n+1}[f]=v_1f(-1)+v_{n+1}f(1)+\sum_{i=2}^n v_if(y_i)\,,
\]
where $y_2,\dots,y_{n}\in(-1,1)$ are (distinct) roots of $P_n'$ and
\[
 v_1=v_{n+1}=\frac{2}{n(n+1)}\,,\quad v_i=\frac{2}{n(n+1)P_n^2(y_i)}\,,\quad i=2,\dots,n\,.
\]
For these forms of quadratures as well as for another quadratures appearing in this paper see for instance~\cite{Hil87}.

Recall that a~continuous function $f:[-1,1]\to\R$ is $n$-\emph{convex} ($n\in\N$), if and only if $f$ is of the class $\C^{n-1}$ and the derivative $f^{(n-1)}$ is convex (\cf~\cite[Theorem~15.8.4]{Kuc09}). For the needs of this paper it could be regarded as a~definition of $n$-convexity.

Let $\T$ be a~positive linear functional defined (at least) on a~linear subspace of $\R^{[-1,1]}$ generated by the cone of $(2n-1)$-convex functions (\ie~$\T[f]\xge 0$ for $f\xge 0$). Assume that $\T$ is exact on polynomials of order $2n-1$, \ie~$\T[p]=\int_{-1}^1 p(x)\dx$ for any polynomial~$p$ of order $2n-1$. It was proved in~\cite[Theorem~14]{Was10MIA} that the inequality
\begin{equation}\label{eq:extr}
 \G_n[f]\xle\T[f]\xle\Lob_{n+1}[f]
\end{equation}
holds for any $(2n-1)$-convex function $f:[-1,1]\to\R$. Then the functionals $\G_n$ and $\Lob_{n+1}$ restricted to the cone of $(2n-1)$-convex functions are minimal and maximal, respectively, among all positive linear functionals defined (at least) on $(2n-1)$-convex functions, which are exact on polynomials of order $2n-1$. In~\cite[Theorem~15]{Was10MIA} there is a~counterpart of the above result for $2n$-convex functions with Radau quadratures in the role of the minimal and maximal operators.

Studying the results of this kind the following problem seems to be natural. Some quadrature operators are extremal in the sense of inequalities like~\eqref{eq:extr}. On the other hand, the set of all quadratures which are exact on polynomials of some given order is convex. Then it could be interesting to find its extreme points looking for the possible connections between extremalities in the approximate integration and the extreme points of convex sets. In particular, are $\G_n$ and $\Lob_{n+1}$ extreme points of the above mentioned set? If the answer is positive, are they the only extreme points, or there exist another ones?

This is the starting point for our considerations. We will observe that the extreme points in the set of all quadratures exact on polynomials of prescribed order could be determined with the aid of~\cite[Theorem~6.1, p.~101]{KarStu66}. Next we shall investigate the extreme points the set of all positive linear operators defined on $\C[-1,1]$ with prescribed moments. Our research is far from being complete. Actually we are able to give a~full description of the extreme points of the set of symmetric operators with four prescribed moments, \ie~$(m_0,m_1,m_2,m_3)=(1,0,b^2,0)$.

\section{Extremal quadratures}

Let $D$ be a~convex subset of a~linear space. Recall that $x\in D$ is the extreme point of~$D$, if $x$ is not the "interior" point of any segment with endpoints in~$D$, \ie~$x=tu+(1-t)v$ for some $u,v\in D$ and $t\in[0,1]$ implies that $x=u=v$. The set of all extreme points of a~set $D$ will be denoted by $\ext D$.

A~\emph{quadrature} on $[-1,1]$ is the linear functional defined on $\R^{[-1,1]}$ by the formula
\[
 \QQ[f]=\sum_{k=1}^{n(\QQ)}w^{\QQ}_kf(\xi^{\QQ}_k)\,,
\]
where $n(\QQ)\in\N$, $\xi^{\QQ}_k\in[-1,1]$ are the \emph{nodes} and $w^{\QQ}_k$ are the \emph{weights} of $\QQ$ (for $k=1,\dots,n(\QQ)$). If all the weights of $\QQ$ are positive, then $\QQ$ is a~\emph{positive quadrature}, \ie\linebreak $\QQ[f]\xge 0$ for $f\xge 0$. Positive quadratures are often used in the approximate integration.


Let $e_k(x)=x^k$, $k=0,1,\dots,n$. Fix a~vector $\mathbf{\m}=(m_0,m_1,\dots,m_n)\in\R^{n+1}$. Let $\kw(\m)$ be the~set of all positive quadratures $\QQ$ with moments $\QQ[e_k]=m_k$, $k=0,1,\dots,n$.  In this section we determine the extreme points of the (convex) set $\kw(\m)$. In particular, if $\m$ is a~vector of integral moments $m_k=\int_{-1}^1x^k\dx$, $k=0,1,\dots,n$, we will obtain a connection between the extremalities in the approximate integration and extreme points of a~set of positive quadratures, which are exact on polynomials of a~given order~$n$.

Every positive quadrature $\QQ$ could be written in the form
\[
 \QQ[f]=\int_{[-1,1]}f\dx[\mu_{\QQ}]\,\quad\text{for } \mu_{\QQ}=\sum_{k=1}^{n(\QQ)}w^{\QQ}_k\delta_{\xi^{\QQ}_k},
\]
where $\delta_x$ stands for a~Dirac measure concentrated at~$x$. By the Riesz--Markov Theorem (\cf~\cite[p.~458]{RoyFit10}) the measure $\mu_{\QQ}$ in the above representation is uniquely determined. Furthermore, if $\QQ\in\kw(\m)$, then $\m$ is the moment vector of the~measure $\mu_{\QQ}$.

Denote by $\D(\m)$ the set of all discrete measures $\mu$ on $\B\bigl([-1,1]\bigr)$ with moments
\[
 \int_{[-1,1]}e_k\dx[\mu]=m_k\,,\quad k=0,1,\dots,n\,.
\]
The set $\D(\m)$ is convex.

\begin{thm}
 A~quadrature $\QQ\in\kw(\m)$ is an extreme point of $\kw(\m)$ if and only if $n(\QQ)\xle n+1$.
\end{thm}

\begin{proof}
 A quadrature $\QQ\in\kw(\m)$ is an extreme point of $\kw(\m)$ if and only if the measure $\mu_{\QQ}\in\D(\m)$ is the extreme point of $\D(\m)$. By virtue of \cite[Theorem~6.1, p.~101]{KarStu66} the extreme measures in $\D(\m)$ are exactly the measures concentrated on at most $n+1$ points. This finishes the proof.
\end{proof}

For $m_k=\int_{-1}^1x^k\dx$, $k=0,1,\dots,2n-1$, we obtain immediately that the $n$-point Gauss quadrature $\G_n$, as well as the $(n+1)$-point Lobatto quadrature, are the extreme points of $\kw(\m)$. Nevertheless, there are infinitely many other extremal quadratures in this set. For instance, all Gauss quadratures and with $p$ nodes ($p\in\{n,\dots,2n\}$), also all Lobatto quadratures with the number of nodes $p\in\{n+1,\dots,2n\}$, are the extreme points of $\kw(\m)$.

\section{Extremal measures}

To find extremal quadratures we needed to know the extreme points of a~set of all discrete measures with finite spectrum. In this section we consider all finite symmetric measures on $\B\od$ with prescribed moments $(1,0,b^2,0)$.

Let $\M\od$ be the~set of all finite measures on $\B\od$. Let $\M^0\od$ be a~subset of $\M\od$ consisting of measures, which are symmetric \wrt~$0$, \ie~$\mu\in\M^0\od$ if and only if $\mu\in\M\od$ and $\mu(-B)=\mu(B)$, $B\in\B\od$.

Let $\P\od$, $\P(\R)$ be the sets of probability measures on $\B\od$ and $\B(\R)$, respectively. Denote

\[
 \P^0\od=\M^0\od\cap\P(\R)\,. 
\]

For a~non-zero measure $\mu\in\M\od$ define the measure $\tmu$ by
\begin{equation*}
 \tmu(B)=\frac{\mu(B)}{\mu\od}\,,\quad B\in\B\od\,.
\end{equation*}

Let $0<a<1$ and $\M^0\oda$ be the set of all measures $\mu\in\M^0\od$ satisfying
\begin{equation}\label{eq:2}
 \int_{-1}^1 x^2\mu(\dx)=a^2\mu\od\,.
\end{equation}

Set
\[
 \P^0\oda=\M^0\oda\cap\P\od\,.
\]
Clearly
\begin{equation*}
 \int_{-1}^1 x^2\mu(\dx)=a^2\,.
\end{equation*}
for any $\mu\in\P^0\oda$. Moreover,
\[
 \mu\in\M^0\oda\iff \tmu\in\P^0\oda\,,
\]
whenever $\mu$ is a~non-zero measure.

Obviously, the set $\P^0\odb$ is the set consisting of all finite symmetric measures on $\B\od$ with prescribed moments $(1,0,b^2,0)$.


We start with two lemmas. The proof of the first of them is rather standard ad simple, so we omit it.
\begin{lem}\label{lem:1}
 Let $\mu\in\P^0\od$ and $m_2=\displaystyle\int_{-1}^1 x^2\mu(\dx)$ for $0<a<1$.
 \begin{enumerate}[\upshape a)]
  \item
   If $\mu$ is concentrated on~$[-a,a]$, then $m_2\xle a^2$.
  \item
   If $\mu$ is concentrated on~$[-1,-a]\cup[a,1]$, then $m_2\xge a^2$.
  \item
   If $\mu=\dfrac{\delta_{-a}+\delta_a}{2}$, \ie~$\mu$ is concentrated on the set $\{-a,a\}$, then $m_2=a^2$.
  \item 
   If $\mu$ is concentrated on $[-a,a]$ and $\mu\bigl((-a,a)\bigr)>0$, then $m_2<a^2$.
  \item 
   If $\mu$ is concentrated on $[-1,-a]\cup[a,1]$ and $\mu\bigl((-1,-a)\cup(a,1)\bigr)>0$, then $m_2>a^2$.
  \item
   Suppose that $\mu$ is concentrated on $[-a,a]$. Then $m_2=a^2$ if and only if $\mu$ is concentrated on $\{-a,a\}$.
  \item
   Suppose that $\mu$ is concentrated on $[-1,-a]\cup[a,1]$. Then $m_2=a^2$ if and only if $\mu$ is concentrated on $\{-a,a\}$. 
 \end{enumerate}
\end{lem}

Let $\mu\in\M\od$, $E\in\B\od$. Then $\mu_E$ stands for the restriction of $\mu$ to the set $E$, \ie~$\mu_E(B)=\mu(B\cap E)$, $B\in\B\od$. Similarly to \eqref{eq:2}, for any $E\in\B\od$ with $\mu(E)>0$, we put
\[
 \tmu_E(B)=\frac{\mu_E(B)}{\mu(E)}\,,\quad B\in\B\od\,.
\]

\begin{lem}\label{lem:2}
 Let $0<a<1$ and $\mu\in\M^0\oda$ be a~non-zero continuous measure, \ie~$\mu\bigl(\{x\}\bigr)=0$ for all $x\in[-1,1]$. Then
 \begin{enumerate}[\upshape (i)]
  \item
   $\mu\bigl((-a,a)\bigr)>0$ and $\mu\bigl((-1,-a)\cup(a,1)\bigr)>0$;
  \item
   there exists a~$\xi_0\in(a,1)$ \st~$\mu\bigl((a,\xi_0)\bigr)>0$ and $\mu\bigl((\xi_0,1)\bigr)>0$.
 \end{enumerate}
\end{lem}
\begin{proof}
 The part (i) follows immediately by Lemma~\ref{lem:1}, because otherwise either $m_2>a^2$, or $m_2<a^2$. To prove (ii) assume, on the contrary, that for all $\xi\in(a,1)$,
 \begin{equation}\label{eq:4}
  \mu\bigl([a,\xi]\bigr)=0\quad\text{or}\quad\mu\bigl([\xi,1]\bigr)=0\,.
 \end{equation}
 We recursively define the sequence of sets $A_n=[a_n,b_n]\subset[a,1]$, $n\in\N$ starting with $A_1=[a,1]$. Using (i) and taking into account the symmetry of $\mu$, we get $\mu\bigl([a,1]\bigr)>0$. Suppose that we have constructed the sets $A_k$, $k=1,2,\dots,n$ \st~$\mu(A_k)=\mu\bigl([a,1]\bigr)$, $k=1,2,\dots n$ and $A_k\subset A_{k-1}$, $k=2,\dots,n$. If $\xi_n=\frac{a_n+b_n}{2}$, then, by~\eqref{eq:4}, two cases are possible. If $\mu\bigl([a,\xi_n]\bigr)=0$, then $\mu\bigl([\xi_n,1]\bigr)=\mu\bigl([a,1]\bigr)$ and we take $a_{n+1}=\xi_n$, $b_{n+1}=b_n$. If $\mu\bigl([\xi_n,1]\bigr)=0$, then $\mu\bigl([a,\xi_n]\bigr)=\mu\bigl([a,1]\bigr)$ and we take $a_{n+1}=a_n$, $b_{n+1}=\xi_n$. Obviously, for $A_{n+1}=[a_{n+1},b_{n+1}]$ we have $A_{n+1}\subset A_n$ and $\mu(A_{n+1})=\mu\bigl([a,1]\bigr)$.
 
 By the above construction $\mu\left(\displaystyle\bigcap_{n=1}^{\infty}A_n\right)=\mu\bigl([a,1]\bigr)>0$ and there exists $x\in[a,1]$ \st~$\displaystyle\bigcap_{n=1}^{\infty}A_n=\{x\}$. Because $\mu$ was continuous, we arrive at the contradiction. This completes the proof.
\end{proof}

Below we prove some decomposition-type result.

\begin{thm}\label{th:3}
 Let $0<a<1$ and $\mu\in\P^0\oda$ be a~continuous measure. There exist the sets $E_1,E_2\in\B\od$ \st~$E_1\cap E_2=\emptyset$, $\mu(E_1)>0,\mu(E_2)>0$, $\mu(E_1)+\mu(E_2)=1$ and $\mu|_{E_1},\mu|_{E_2}\in\M^0\oda$.
\end{thm}
\begin{proof}
 Since $\mu$ is continuous and symmetric, Lemma~\ref{lem:2} implies
 \begin{equation*}
  \mu\bigl((0,a)\bigr)>0\,,\quad \mu\bigl((a,1)\bigr)>0\,,\quad\mu\bigl((0,a)\cup(a,1)\bigr)=\frac{1}{2}\,.
 \end{equation*}
 Consider the function $g:[a,1]\to\R$ given by
 \begin{equation*}
  g(x)=\int_{-1}^1 u^2\tmu|_{[-x,x]}(\dx[u])\,,\quad a\xle x\xle 1
 \end{equation*}
 and the measure $\nu=a^2\mu$, which is absolutely continuous \wrt~$\mu$. In particular, $\nu$ is continuous. Denote by $F_{\mu},F_{\nu}$ the distribution functions of the measures $\mu,\nu$, respectively, and rewrite the function $g$ in the form
 \begin{align}\label{eq:7}
  g(x)&=\biggl(\int_{-1}^1 u^2\mu|_{[-x,x]}(\dx[u])\biggr)\cdot\Bigl(\mu\bigl([-x,x]\bigr)\Bigr)^{-1}\\\nonumber
  &=\biggl(\int_{-x}^x u^2\mu(\dx[u])\biggr)\cdot\Bigl(\mu\bigl([-x,x]\bigr)\Bigr)^{-1}=\frac{\nu\bigl([-x,x]\bigr)}{\mu\bigl([-x,x]\bigr)}=\frac{F_{\nu}(x)-F_{\nu}(-x)}{F_{\mu}(x)-F_{\mu}(-x)}\,.
 \end{align}
 The distributions functions $F_{\mu},F_{\nu}$ are continuous by continuity of the measures $\mu,\nu$, respectively. Furthermore,
 \[
  F_{\mu}(x)-F_{\mu}(-x)=\mu\bigl([-x,x]\bigr)\xge\mu\bigl([-a,a]\bigr)>0
 \]
 for all $x\in[a,1]$. Hence $g$ is continuous by~\eqref{eq:7}.
 
 By Lemma~\ref{lem:1} we infer that $g(a)<a^2$. Since $\mu\in\P^0\oda$, then $g(1)=a^2$. Continuity of~$g$ implies that there exists $b_1\in(a,1)$ \st
 \begin{equation}\label{eq:8}
  g(a)<g(b_1)\quad\text{and}\quad g(b_1)<a^2\,.
 \end{equation}
 Then
 \begin{equation}\label{eq:9}
  \mu\bigl((a,b_1)\bigr)>0\quad\text{and}\quad\mu\bigl((b_1,1)\bigr)>0\,.
 \end{equation}
 Indeed, suppose that $\mu\bigl((a,b_1)\bigr)=0$. Hence $\mu|_{(a,b_1)}$ is a~zero-measure and consequently
\begin{align*}
  g(b_1)&=\biggl(\int_{-1}^1 u^2\mu|_{[-b_1,b_1]}(\dx[u])\biggr)\cdot\Bigl(\mu\bigl([-b_1,b_1]\bigr)\Bigr)^{-1}\\
  &=\biggl(\int_{-1}^1 u^2\mu|_{[-a,a]}(\dx[u])\biggr)\cdot\Bigl(\mu\bigl([-a,a]\bigr)\Bigr)^{-1}=g(a)\,,
 \end{align*}
 which contradicts~\eqref{eq:8}. Similarly, supposing that $\mu\bigl((b_1,1)\bigr)=0$, we arrive at $g(b_1)=g(1)=a^2$, which also contradicts~\eqref{eq:8}.
 
 Now we define the function $h:[0,a]\to\R$ by
 \[
  h(x)=\int_{-1}^1 u^2\tmu|_{[-b_1,-x]\cup[x,b_1]}(\dx[u])\,,0\xle x\xle a\,.
 \]
 Writing
 \[
  h(x)=\frac{\nu\bigl([x,b_1]\bigr)}{\mu\bigl([x,b_1]\bigr)}=\frac{F_{\nu}(b_1)-F_{\nu}(x)}{F_{\mu}(b_1)-F_{\mu}(x)}
 \]
 and using once more the continuity of $F_{\mu},F_{\nu}$, we obtain that $h$ is continuous. We have $h(0)=g(b_1)<a^2$ and
 \[
  h(a)=\int_{-1}^1 u^2\tmu|_{[-b_1,-a]\cup[a,b_1]}(\dx[u]).
 \]
 Using Lemma~\ref{lem:1} we arrive at $h(a)>a^2$. Consequently, by continuity of~$h$ we conclude that there exists $a_1\in(0,a)$ \st
 \begin{equation}\label{eq:10}
  h(a_1)=a^2\,.
 \end{equation}
 Define $E_1=[-b_1,-a_1]\cup [a_1,b_1]$, $E_2=[-1,1]\setminus E_1$. By~\eqref{eq:10} we obtain
 \begin{equation}\label{eq:11}
  \int_{-1}^1 u^2\mu|_{E_1}(\dx[u])=a^2\mu(E_1)\,.
 \end{equation}
 Taking into account the above equation and the moment condition $\int_{-1}^1 u^2\mu(\dx[u])=a^2$, we have
 \begin{equation}\label{eq:12}
 \int_{-1}^1 u^2\mu|_{E_2}(\dx[u])=a^2\mu(E_2)\,.
 \end{equation}
 The following inequalities are true:
 \begin{equation}\label{eq:13}
  \mu\bigl((0,a_1)\bigr)>0\quad\text{and}\quad \mu\bigl((a_1,a)\bigr)>0\,.
 \end{equation}
 Indeed, suppose that $\mu\bigl((0,a_1)\bigr)=0$. Then $h(a_1)=h(0)<a^2$, which contradicts~\eqref{eq:10}. Similarly, if $\mu\bigl((a_1,a)\bigr)=0$, then $h(a_1)=h(a)>a^2$, which also contradicts~\eqref{eq:10}.
 
 By~\eqref{eq:9} and~\eqref{eq:13} we infer that $\mu(E_1)>0$ and $\mu(E_2)>0$. Using~\eqref{eq:11} and~\eqref{eq:12} we arrive at $\mu|_{E_1},\mu|_{E_2}\in\M^0\oda$. The proof is now complete.
\end{proof}

Our next result offers some decomposition of a~continuous measure.

\begin{thm}\label{th:4}
 Let $0<a<1$ and $\mu\in\P^0\oda$ be a~continuous measure.
 \begin{enumerate}[\upshape (i)]
  \item
   There exist non-zero measures $\mu_1,\mu_2\in\M^0\oda$ such that $\mu=\mu_1+\mu_2$ and $\mu_1\ne c_1\mu$, $\mu_2\ne c_2\mu$ for any $c_1,c_2>0$.
  \item
   There exist measures $\nu_1,\nu_2\in\P^0\oda$ and $\alpha\in(0,1)$ \st~$\nu_1\ne\mu$, $\nu_2\ne\mu$ and
   \begin{equation}\label{eq:14}
    \mu=\alpha\nu_1+(1-\alpha)\nu_2\,.
   \end{equation}
 \end{enumerate}	 
\end{thm}	
\begin{proof}
 \begin{enumerate}[\upshape (i)]
  \item
   Take $E_1,E_2\in\B\od$ given by Theorem~\ref{th:3} and denote $\mu_1=\mu|_{E_1}$, $\mu_2=\mu|_{E_2}$. Since the measures $\mu_1,\mu_2$ are concentrated on disjoint sets and $\mu=\mu_1+\mu_2$, we conclude that $\mu_1\ne c_1\mu$, $\mu_2\ne c_2\mu$ for any $c_1,c_2>0$.
  \item
   Put $\nu_1=\tmu_1$, $\nu_2=\tmu_2$, where $\mu_1,\mu_2$ are defined in~(i). Since $\mu_1\ne c_1\mu$ and $\mu_2\ne c_2\mu$ for any $c_1,c_2>0$, then $\nu_1\ne\mu$, $\nu_2\ne\mu$. Setting $\alpha=\mu_1(E)$, we get~\eqref{eq:14}. This finishes the proof.
 \end{enumerate}	
\end{proof}

The corollary below follows immediately by~Theorem~\ref{th:4}.

\begin{cor}\label{cor:5}
 Let $0<a<1$ and $\mu\in\P^0\oda$. If $\mu$ is a~continuous measure, then $\mu$ is not the extreme point of $\P^0\oda$.
\end{cor}

Let $0<b<1$. By $\K(b)$ we denote the set of all discrete symmetric probability measures~$\mu$ on $\B\od$ with prescribed moments $(1,0,b^2,0)$ and admitting at most~four mass points. Now we state for the symmetric probability measure~$\mu$ a~necessary condition to be the extreme point of $\P^0\odb$.

\begin{thm}\label{th:6}
 Let $0<b<1$. Then  
  $\ext\Bigl(\P^0\odb\Bigr) \subset \K(b).$ 
\end{thm}

\begin{proof}
 Let $\sigma\in\P^0\odb$ be the extreme point of $\P^0\odb$. Then $\sigma$ can be uniquely represented as the sum of a~continuous measure and a discrete measure:
 \begin{equation}\label{eq:15}
  \sigma=\beta\lambda_1+(1-\beta)\lambda_2\,,
 \end{equation}
 where $\beta\in[0,1]$, $\lambda_1,\lambda_2\in\P\od$, $\lambda_1$ is a~continuous measure, while $\lambda_2$ is a~discrete measure.
 
 Observe that $0\xle\beta<1$, because if $\beta=1$, then $\sigma$ was continuous and, by Corollary~\ref{cor:5}, $\sigma$ was not the extreme point. If $\beta=0$, then $\sigma$ is a~discrete measure and the assertion follows by \cite[Theorem 6.1, p.~101]{KarStu66}.
 
 Suppose now that $0<\beta<1$. We claim that in this case $\sigma$ is not the extreme point. If we show it, the proof is finished.
 
 It is not difficult to check that both $\lambda_1$ and $\lambda_2$ are symmetric \wrt~$0$. Indeed, if $\sigma\bigl(\{x\}\bigr)>0$ for some $x\in[-1,1]$, then $\sigma\bigl(\{-x\}\bigr)>0$. Therefore a~discrete part of $\sigma$, \ie~$(1-\beta)\lambda_2$, is symmetric, which implies that $\lambda_2$ is symmetric. Hence also $\lambda_1$ is symmetric.
 
 As a~probability measure, $\lambda_1$ is non-zero. Then $\lambda_1\in\P^0\oda$, where $a^2=\displaystyle\int_{-1}^1 x^2\lambda_1(\dx)$ and $0<a<1$. Now we apply Theorem~\ref{th:4}~(ii) to the measure $\lambda_1$. There exist the measures $\nu_1,\nu_2\in\P^0\oda$ and $0<\alpha<1$ \st~$\nu_1\ne\nu_2$ and
 \begin{equation}\label{eq:16}
  \lambda_1=\alpha\nu_1+(1-\alpha)\nu_2\,.
 \end{equation}
 
 We have also $\lambda_2\in\P^0\bigl([-1,1],c\bigr)$, where $c^2=\displaystyle\int_{-1}^1 x^2\lambda_2(\dx)$ and $0<c<1$. Write
 \begin{equation}\label{eq:17}
  \lambda_2=\alpha\lambda_2+(1-\alpha)\lambda_2\,.
 \end{equation}
 
 Using the properties of the measures $\sigma,\lambda_1,\lambda_2$ we get
 \begin{equation}\label{eq:18}
  b^2=\int_{-1}^1 x^2\sigma(\dx)=\beta\int_{-1}^1 x^2\lambda_1(\dx)+(1-\beta)\int_{-1}^1 x^2\lambda_2(\dx)=\beta a^2+(1-\beta)c^2.
 \end{equation}
 
 By \eqref{eq:15}, \eqref{eq:16}, \eqref{eq:17}, $\sigma$ can be written as
 \[
  \sigma=\alpha\sigma_1+(1-\alpha)\sigma_2
 \]
 with $\sigma_1=\beta\nu_1+(1-\beta)\nu_2$ and $\sigma_2=\beta\nu_2+(1-\beta)\lambda_2$. Since $\nu_1\ne\nu_2$, then $\sigma_1\ne\sigma_2$. Using \eqref{eq:18} we arrive at $\sigma_1,\sigma_2\in\P^0\odb$, which implies that $\sigma$ is not the extreme point of $\P^0\odb$ and completes the proof.
\end{proof}

Notice that every probability measure~$\mu \in \K(b)$ (with $0<b<1$) can be written in the form
\begin{equation}\label{eq:19}
 \mu_{(x,y)}=\frac{1}{2}p\bigl(\delta_x+\delta_{-x}\bigr)+\frac{1}{2}q\bigl(\delta_y+\delta_{-y}\bigr)
\end{equation}
for some $0\xle x\xle y\xle 1$ and $p,q>0$ with $p+q=1$.

Observe that if
\[
 \int_{-1}^1 u^2\mu_{(x,y)}(\dx[u])=b^2\,,
\]
then
\begin{equation}\label{eq:21}
 px^2+qy^2=b^2\,.
\end{equation}

It is easy to see that if $0\xle x\xle y\xle 1$ and $p,q>0$ with $p+q=1$ satisfying the above equation, then $0 \xle x\xle b\xle y \xle 1$.

If additionally $x\ne y$, then the numbers $p,q$ could be computed by
\begin{equation}\label{eq:21a}
 p=\frac{y^2-b^2}{y^2-x^2}\,,\quad q=\frac{b^2-x^2}{y^2-x^2}\,.
\end{equation}

\begin{rem}\label{rem:7}
The set $\K(b)$ consists of all probability measures $\mu_{(x,y)}$ given by~\eqref{eq:19}, where $0 \xle x \xle b\xle y\xle 1$, $p,q>0$ with $p+q=1$ satisfying~\eqref{eq:21}. 
\end{rem}

 One could easily show the lemma.

\begin{lem}\label{lem:7c}
Let $0<b<1$. Then $\K(b)  \subset \ext\Bigl(\P^0\odb\Bigr).$ 
\end{lem}

Thus, we derive from Theorem~\ref{th:6} and Lemma~\ref{lem:7c} the main result of this paper.

\begin{thm}\label{th:7}
 Let $0<b<1$. Then $\ext\Bigl(\P^0\odb\Bigr)=\K(b).$
\end{thm}

Of course the extreme points of $\P^0\odb$ are the measures admitting $2,3$ or~$4$ mass points. The only two-point extreme measure is
\[
 \mu_{(b,b)}=\frac{1}{2}\bigl(\delta_{-b}+\delta_b\bigr)\,.
\]
All three-point extreme measures have the form
\[
 \mu_{(0,y)}=p\delta_0+\frac{1}{2}q\bigl(\delta_{-y}+\delta_y\bigr)\,,
\]
where $b\xle y\xle 1$. In particular, for $y=1$ we get
\[
 \mu_{(0,1)}=(1-b^2)\delta_0+\frac{b^2}{2}\bigl(\delta_{-1}+\delta_1\bigr)\,.
\]

\section{Integral representation of probability measures}

As an application of results obtained in the previous section concerning the extreme measures we shall give a theorem on integral representation of probability measures from the set $\P^0\odb$.

The set of probability measures $\P(\R)$ is metrizable. In metrizing of the weak convergence of probability measures on $\B(\R)$ the L\'{e}vy-Prohorov distance  (see~\cite[Chapter~11, Theorem 11.3.3, p. 395]{Dud02}) can be used. The set $\P\od$ is a metrizable compact convex subset of $\P(\R)$. 
 It is not difficult to prove that $\P^0\odb$ is a closed and convex subset of $\P\od$. Consequently, we have the following lemma. 
 \begin{lem}
 The set $\P^0\odb$ is a metrizable compact convex subset of $\P\od$.
 \end{lem}
  
On $\P^0\odb$ consider the topology induced from $\P(\R)$ and the mapping $T:[0,b]\times[b,1]\to \K(b)$ given by 
\[
 T(x,y)=\mu_{(x,y)}\,.
\]
Taking into account Remark \ref{rem:7} and the formulae \eqref{eq:21a}, it is not difficult to prove the lemma.

\begin{lem}\label{lem:7b}
 The mapping~$T$ is a~homeomorphism between $[0,b]\times[b,1]$ and $\K(b)$.
 \end{lem}


Now we state the main result of this section.

\begin{thm}\label{th:8}
 Let $0<b<1$. For every probability measure $\sigma\in\P^0\odb$ there exists a~probability measure $\gamma$ on $\B\bigl([0,b]\times[b,1]\bigr)$ \st
 \[
  \int_{-1}^{1} f(u)\sigma(\dx[u])=\int_{-1}^{1}f(u)\biggl(\int_{[0,b]\times[b,1]}\mu_{(x,y)}\gamma\bigl(\dx[(x,y)]\bigr)\biggr)(\dx[u])
 \]
 for any continuous function $f:[-1,1]\to\R$.
\end{thm}
\begin{proof}
We shall use Choquet's Representation Theorem (\cite[p.~14]{Phe01}). Taking into account Theorem~\ref{th:7} we obtain that for every measure $\sigma\in\P^0\odb$ there exists a~probability measure $m\in\K(b)$ \st
\[
 L(\sigma)=\int_{\K(b)} L(\mu)m(\dx[\mu])
\]
for any continuous linear functional $L$ on $\P\od$.

Let $\gamma$ be the~measure on $\B\bigl([0,b]\times[b,1]\bigr)$ defined by $\gamma=m\circ T$, \ie~$\gamma(B)=m(TB)$ for $B\in\B\bigl([0,b]\times[b,1]\bigr)$. 
Taking into account Lemma ~\ref{lem:7b}, for any continuous linear functional $L$ on $\P\od$ we have
\begin{multline}\label{eq:23}
 L(\sigma)=\int_{\K(b)} L(\mu)m(\dx[\mu])=\int_{T^{-1}\bigl(\K(b)\bigr)}L\circ T\bigl((x,y)\bigr)(m\circ T)\bigl(\dx[(x,y)]\bigr)\\
 =\int_{[0,b]\times[b,1]}L\circ T\bigl((x,y)\bigr)\gamma\bigl(\dx[(x,y)]\bigr)\,.
\end{multline}

For every continuous function $f:[-1,1]\to\R$ the linear functional
\[
 L_f(\mu)=\int_{-1}^1 f(u)\mu(\dx[u])
\]
is continuous. Hence, for the probability measure $\sigma\in\P^0\odb$ we obtain by~\eqref{eq:23}
\begin{multline*}
 \int_{-1}^{1} f(u)\sigma(\dx[u])=\int_{[0,b]\times[b,1]}\int_{-1}^{1} f(u)\mu_{(x,y)}(\dx[u])\gamma\bigl(\dx[(x,y)]\bigr)\\
 =\int_{-1}^{1}f(u)\biggl(\int_{[0,b]\times[b,1]}\mu_{(x,y)}\gamma\bigl(\dx[(x,y)]\bigr)\biggr)(\dx[u])\,.
\end{multline*}
The theorem is proved.
\end{proof}

\begin{rem}
 Notice that Theorem~\ref{th:8} is related to \cite[Theorem 6.3, p.~103]{KarStu66}.
\end{rem}

\bibliographystyle{plain}

\end{document}